\documentclass[12pt]{article}

\usepackage[shrink=50,letterspace=500]{microtype}
\usepackage[utf8]{inputenc}

\usepackage[english]{babel}

\usepackage[a4paper,left=2.5cm,top=3cm,textwidth=16cm,textheight=24.5cm]{geometry}

\usepackage[bookmarksopen=true,bookmarks=true,unicode,setpagesize]{hyperref}
\hypersetup{colorlinks=true,linkcolor=black,citecolor=black}

\usepackage[intlimits]{amsmath}
\usepackage{mathtools}
\usepackage{mathrsfs}
\usepackage{stackrel}
\usepackage{yhmath}
\usepackage{enumerate}
\usepackage{footnpag}
\usepackage[noadjust]{cite}

\usepackage{amsthm,amssymb}

\theoremstyle{plain}
\newtheorem{theorem}{Theorem}

\newtheorem{proposition}[theorem]{Proposition}

\theoremstyle{definition}

\newtheorem{example}[theorem]{Example}

\theoremstyle{remark}
\newtheorem{remark}[theorem]{Remark}

\allowdisplaybreaks[3]

\newcommand\B{\mathcal{B}}

\renewcommand\G{\mathcal{G}}

\newcommand\N{\mathbb{N}}

\newcommand\R{\mathbb{R}}

\newcommand\la{\lambda}

\newcommand\cnt[2]{\text{\setbox2=\hbox{#1}\rlap{\hbox to \wd2{\hfil#2\hfil}}\box2}}

\newcommand\X{\R^d}

\newcommand\ind[1]{\relax}

\newcommand\term[3]{\relax}

        {\end{list}}

\makeatother

\begin{document}

\pagestyle{plain}
\title{Random potentials for Markov processes}

\date{}
\author{
\textbf{Yuri Kondratiev}\\
 Department of Mathematics, University of Bielefeld, \\
 D-33615 Bielefeld, Germany,\\
 Dragomanov University, Kyiv, Ukraine\\
 e-mail: kondrat@mathematik.uni-bielefeld.de\\
 Email: kondrat@math.uni-bielefeld.de 
\and\textbf{ Jos{é} L.~da Silva},\\
 CIMA, University of Madeira, Campus da Penteada,\\
 9020-105 Funchal, Portugal.\\
e-mail: joses@staff.uma.pt}
\date{\today}

\maketitle
\begin{abstract}
The paper is devoted to the integral functionals $\int_0^\infty f(X_t)\,{\mathrm{d}t}$ of Markov processes in $\X$ in the case  $d\ge 3$. It is established that such functionals can be presented as the integrals $\int_{\X} f(y) \G(x, \mathrm{d}y, \omega)$
with  vector valued random measure $\G(x, \mathrm{d}y, \omega)$. Some examples such as compound Poisson processes, Brownian motion and diffusions  are considered. 


\end{abstract}

\tableofcontents{}

\section{Introduction}
Let  $X=\{X(t), t\ge 0\}$ be a Markov process in $\X$ starting from $x\in\X$.
For a function $f:\X\to \R$ the potential of $f$ is defined as \cite{GB}
$$
u_f(x)= \int_0^\infty E^x [f(X(t)]\,\mathrm{d}t.
$$
The existence of the potential $u_f(x)$ is a difficult question and the class
of admissible $f$ shall be analyzed for each process separately. 
An alternative approach is based on the use of the generator $L$
of the process.  Namely, the potential $u_f$ may be constructed
as the solution to the following equation:
$$
-Lu=f.
$$
Of course, there appear a technical problem of the characterization of
the domain of the inverse generator $L^{-1}$.  In the analogy with
the PDE framework, we would like to have a representation
$$
u_f(x) =\int_{\X} f(y) \G(x,\mathrm{d}y),
$$
where $\G(x,dy)$ is a measure on $\X$.  This measure is nothing but the fundamental
solution to the considered equation and traditionally may be called the Green measure
for the operator $L$.

Another possibility to study potentials is related with the following observation.
A standard way to define a   homogeneous  Markov process is to give the probability
$P_t(x,B)$ of the transition from the point $x\in \X$  to the set
$B\subset \X$ in   time $t>0$.  In some cases we have
$$
P_t(x,B)=\int_{B} p_t(x,y)\,\mathrm{d}y,
$$
where $p_t(x,y)$ is the density of the transition probability.
The function
$$
g(x,y) =\int_0^\infty p_t(x,y)\,\mathrm{d}t
$$
is called the Green function. Of course, the existence of the Green
function for a given process or  for a given   transition probability is a non-trivial
fact also. Green functions for different classes of Markov processes are
 traditional objects in probability theory, see, e.g.,
\cite{Gri1}, \cite{Gri2} and references therein.

The existence of the transition density is a condition which is not
always satisfied even for simple classes of Markov processes,
see examples below. Hence, we introduce
the Green measure
$$
\G(x,\mathrm{d}y)= \int_0^\infty P_t(x,\mathrm{d}y)\,{\mathrm{d}t}
$$
assuming the existence of this object as a Radon measure on $\X$,
see \cite{KS}.

We have another  object  related with the Markov process $X$.  Namely, having in mind that $X$ starts from $x\in \X$, for certain class of functions $f:\X\to\R$ introduce the random variables which are reasonable to call random potentials:
$$
Y^x(f)= \int_0^\infty f(X(t))\,\mathrm{d}t.
$$
This is an additive functional but contrary to usual
for $d=1$ we have no normalization. It will be clear from our examples given below   that such
object exists typically for $d\geq 3$. Then
its relation to the Green measure is the following:
$$
E^x [Y^x(f)] =\int_{\X} f(y) \G(x,\mathrm{d}y).
$$
The aim of this paper is to show that the random variables $Y^x(f)$
for certain classes of Markov processes have the representation
$$
Y^x(f)(\omega)= \int_{\X} f(y) \G(x, \mathrm{d}y, \omega)
$$
with  vector valued random measure  which we will call the random Green measure.
This problem was already discussed in \cite{KMS} in the framework of stochastic analysis.
Here we would like to develop purely analytic approach to the construction and analysis
of random Green measures  for Markov processes. As a result, we are dealing with pairs
of related objects: potentials and their representations via Green measures and random 
potentials and their representations via random Green measures.

The plan of the paper is the following. We will analyze two classes of
stochastic  processes in $\X$, $d\geq 3$: continuous time translation invariant
random walks and diffusions, in particular, the  Brownian motion. For each such
class we will show the existence of the random Green measures and establish
some  of their properties. In particular, we will describe a natural class of
admissible function $f$ integrable w.r.t.~the Green measures.

Note that even  for such simple Markov processes as random walks the transition
probability may have a complicated behavior in time-space variables, see
\cite{GKPZ}. But the Green  measures may have nice visible properties. This
is related with an averaging of transition probabilities in the definition
of Green measures. The last effect  is well known in different  models of
dynamics.

\section{Random walks}

\subsection{Jump generators and Green measures}

Let us describe briefly certain results from \cite{KS} which we need
for our considerations.

Let us fix a density kernel $a:\X \to \R$ with the following properties:
$$
a(-x)=a(x), \;\; a\geq 0,\:\; a\in C_b(\X),
$$
$$
\int_{\X} a(y)\,\mathrm{d}y=1.
$$
Consider the generator
$$
Lf(x)= \int_{\X} a(x-y)[f(y)-f(x)]\,\mathrm{d}y.
$$
This operator  can be defined in a proper function space $E$. 
As $E$, we may consider the space of bounded measurable functions
$B(\X)$, the Banach space of bounded continuous functions $C_b(\X)$ or
the Lebesgue spaces $L^p(\X)$, $p\ge 1$, depending on the case.

In particular,
$
L^\ast =L$ in $L^2(\X)$ and  $L$ is a bounded linear operator  in all $L^p(\X)$.
We call this operator a jump generator with the jump
kernel $a$. The corresponding Markov process is of a pure jump type and is known
in stochastic  as  a compound Poisson process \cite{Sk}.

Several analytic properties of jump generators were studied recently in
\cite{GKPZ},  \cite{KMPZ}, \cite{KMV}. We will formulate certain necessary
facts concerning  these operators.

Because $L$ is a convolution operator, it is natural to apply Fourier
transform to  study it.
Consider the Fourier image of the jump kernel

$$
\hat{a}(k)= \int_{\X} e^{-i(k,y)} a(y)\,\mathrm{d}y.
$$
Then
$$
\hat{a}(0) =1,\;\; |\hat{a}(k)| \leq 1,\; k\neq 0,
$$
$$
\hat{a}(k) \to 0,\; k\to \infty.
$$
In the Fourier image   $L$ is the operator of multiplication by the function
$$
\hat{L} (k) = \hat{a}(k)-1,
$$
that is, the symbol of $L$.

In the following we will always assume  that $\hat{a} \in L^1(\X)$
and $a$ has finite second moment, that is,
$$
\int_{\X} |x|^2 a(x)\,\mathrm{d}x <\infty.
$$

The resolvent $R_\la(L)= (\la - L)^{-1}$ for $\la>0$
has a kernel
$$
\G_\la(x,y)= \frac{1}{1+\la}\big(\delta(x-y) + G_\la(x-y)\big)
$$
with
$$
G_\la(x)= \sum_{k=1}^\infty \frac{a_k(x)}{(1+\la)^k},
$$
where
$$
a_k(x) = a^{\ast k}(x)
$$
is the $  k$-fold convolution of the kernel $a$.
As any Radon measure, the Green measure may be considered
as a (translation invariant)  generalized function of the form
$$
\G_0 (x) = \delta (x) + G_0 (x).
$$

The transition probability density  $p(t,x)$   in terms of
Fourier transform has representation
$$
p(t,x)= \frac{1}{(2\pi)^d} \int_{\X} e^{i(k,x)+ t(\hat{a} (k)-1) }\,\mathrm{d}k,
$$
and for the resolvent kernel
$$
\G_\la (x) = -(L-\la)^{-1} (x),
$$
it holds that
$$
\G_\la(x-y)= \frac{1}{(2\pi)^d} \int_{\X} \frac{e^{i(k,x-y)}}{1-\hat{a}(k) +\la}\,\mathrm{d}k.
$$
For a regularization of the last expression we write
$$
\frac{1}{1-\hat{a}(k) +\la}= \frac{1}{1+\la} + \frac{\hat{a}(k)} {(1+\la)(1-\hat{a}(k) +\la)}.
$$
Then for operators we have
$$
\G_\la= \frac{1}{1+\la} + G_\la,
$$
or in the terms of  kernels
$$
G_\la (x-y)= \frac{1}{(2\pi)^d} \int_{\X} \frac{1}{1+\la} \frac{ \hat{a}(k) e^{i(k,x-y)}}{1-\hat{a}(k) +\la}\,\mathrm{d}k.
$$

To summarize our considerations, we note that
the study of Green kernels  is reduced to the analysis of the following objects.
The regular part of the Green kernel is
$$
G_0 (x)= \sum_{k=1}^\infty {a_k(x)},
$$
$$
a_k(x) = a^{\ast k}(x)
$$
$k$-fold convolution.

The Fourier representation for $G_0$:

$$
G_0 (x)= \frac{1}{(2\pi)^d} \int_{\X} \frac{ \hat{a}(k) e^{i(k,x)}}{1-\hat{a}(k)}\,\mathrm{d}k.
$$

For  $d\geq 3$ this integral exists for all $x\in\X$ that follows from the integrable
singularity  of $(1-\hat{a}(k))^{-1}$ at $k=0$. The latter is the consequence of our
assumptions on $a(x)$.

Denote by $X(t)$ our compound Poisson process or continuous time random walk.

Then \cite{KS}
\begin{align}
E^{x}\left[\int_{0}^{\infty}f(X(t))\,\mathrm{d}t\right] & =\int_{0}^{\infty}T(t)f(x)\,\mathrm{d}t\nonumber \\
 & =\int_{\X}f(y)\G(x,\mathrm{d}y)\label{GM}\\
 & =-(L^{-1}f)(x).\nonumber 
\end{align}
Here $\G(x,\mathrm{d}y)$ (if it exists) is called the Green measure for the process $X$.
The class of admissible $f$ shall be discussed separately.

For the  Green function we can write the representation
$$
\G(x,\mathrm{d}y)= \G(x,y)\,\mathrm{d}y,
$$
where $\G(x,y)$ is a positive generalized function.

We know that for our processes hold
$$
\G(x,y)= \delta(x-y) + G(x-y),
$$
where $G(x)$ is the regular part of the Green density. More precisely, we have
$$
G (x)= \frac{1}{(2\pi)^d} \int_{\X} \frac{ \hat{a}(k) e^{i(k,x)}}{1-\hat{a}(k)}\,\mathrm{d}k.
$$
If $a$ has finite second moment,
 for $d\geq 3$ this integral exists for all $x\in\X$ and is a uniformly bounded function.
 Then for continuous $f\in L^1(\X)$  the expression
 $$
 \int_{\X} f(y) \G(x,\mathrm{d}y)
 $$
 is well defined \cite{KS}.

 Consider for $T>0$ the random variables
 $$
 Y(f, T)= \int_0^T f(X(t))\,\mathrm{d}t.
 $$
 Then in the sense of $L^1(P):=L^1(\Omega,P)$ (in expectation)
 $
 \lim_{T\to\infty} Y(f, T)
$
exists and defines a random variable $Y(f)$.
Actually, it will be reasonable to denote $Y^x(f)$ but in our considerations
the starting point $x\in \X$ is fixed.
The proof is obvious. For $f\geq 0$ we can use monotonicity arguments.

What we know is
\begin{equation}
\label{E}
E^x[Y(f)]= f(x) + \int_{\X} f(y)G(x-y)\,\mathrm{d}y.
\end{equation}
In the following we will discuss vector valued Radon measure, see, e.g., \cite{Dobr}.
 These are $L^1(P)$-valued measures which are finite on all bounded Borel sets $A\in \B_b(\X)$.

Introduce the Banach space $CL(\X)= C_b(\X) \cap L^1(\X)$ with the norm
$$
\|f\|_{CL}:= \|f\|_\infty + \|f\|_1:=\sup|f|+\|f\|_{L^1 (\X)}.
$$

\begin{theorem}
\label{thm:representation1}
For each $x\in \X$ the operator
$$
Y: CL(\X) \to L^1 (P)
$$
has a unique representation given, for all $f\in CL(\X)$ and every $\omega\in\Omega$, by
\begin{equation}
\label{repr}
Y(f)(\omega)= \int_{\X} f(y) \mu^x (\mathrm{d}y, \omega)
\end{equation}
with a vector valued  $\sigma$-additive
$($in the strong topology of $L^1(P))$
Radon measure $\mu^x (\mathrm{d}y, \omega)$ on $\B_b(\X)$.
\end{theorem}

\begin{proof}
First of all we note that
the mapping
$$
CL(\X)\ni f \mapsto Y(f) \in L^1(P)
$$
is linear and a continuous operator. It follows from
(\ref{E}) taking into account the boundedness of $G$.
The space $L^1(P)$ is weakly complete
\cite{DS}, Theorem VI.8.6. Then this mapping is
weakly compact, see \cite{Bartle}, Theorem 3.5.
Now we would like  to apply a representation theorem
from \cite{Dobr}, \cite{Kluv}.  Certain technical difficulty here
is related to the standard framework of such type of representation
theorems. The known approaches consider mappings on spaces of continuous
functions (on locally compact spaces) with zero limits at the infinity. 
Our space $CL(\X)$ is a new type of Banach
space of continuous functions which did not appear before
in the general theory.

To obtain the desired representation we proceed as follows. 
At first we take the closed ball $B_N(0) \subset \X$
with radius $N\in\N$ centred at zero. Consider our operator $Y$ on the space
$C(B_N(0))$. Then we may apply the mentioned results to obtain the  representation
$$
\label{rep_N}
Y(f)(\omega)= \int_{B_N(0)} f(x) \mu_N^x(\mathrm{d}y,\omega), \;f\in C(B_N(0)),\;\omega\in\Omega.
$$
Here $\mu_N^x(\mathrm{d}y,\omega)$ is a vector valued Radon measure on
$\B(B_N(0))$ with values in $L^1(\Omega, P)$. The family of measures
$\mu_N^x(\mathrm{d}y,\omega), N\in \N$ is consistent and define the limit  measure by $\mu^x(\mathrm{d}y,\omega)$ on $\B_{b}(\X)$. Using a priori equality (\ref{E}), which
holds for all $f\in C(B_N(0))$ and a standard approximation technique,
the statement of the theorem follows.
\end{proof}

\begin{remark}
The representation (\ref{repr}) gives additional information on the measure $\mu^x (\mathrm{d}y, \omega)$. 
\begin{enumerate}
\item Because (\ref{repr}) is valid for all $f\in CL(\X)$, then we have
$$
\mu^x (dy, \omega)= \delta(x-y) + R(x,y,\omega)\,\mathrm{d}y
$$
with
$$
E[R(x,y,\cdot)] = G(x-y).
$$
\item For every bounded Borel set $A\subset \X$ it holds
$$
\mu^x(A,\cdot)\in L^1(P).
$$
Moreover,  as above, from (\ref{E})
follows
$$
E[\mu^x(\mathrm{d}y, \cdot)]  = \delta(x-y) + G(x-y)\,\mathrm{d}y.
$$
\item In certain particular models (see examples below) we have
integrable functions $G(x)$ and then
$$
E[\mu^x(\X,\cdot) ] <\infty,
$$
i.e., $\mu^x(\X,\omega) < \infty$, $P$-a.a.~$\omega\in\Omega$.
\end{enumerate}
\end{remark}

The next interesting question is to study the Green measure $\mu^x(dy, \cdot)$, namely if it is really a random variable or may be
degenerated in a constant. To this end, we need to calculate the variance of the random variable
$Y^x(f)$ for $x\in\X$ and $f\in CL(\X)$. Without lost of generality and for simplicity of calculations, we take $x=0$
and
\begin{equation}
\label{exp-function}
f(y)= e^{-\sum_{k=1}^d |y_k|}, \;y\in \X.
\end{equation}
From $E[\mu^0(\mathrm{d}y, \cdot)]  = \delta(y) + G(y)\,\mathrm{d}y$ it follows that
$$
\int_0^\infty E[f(X(t))]\,\mathrm{d}t = f(0) + \int_{\X} f(y) G(y)\,\mathrm{d}y.
$$
Introduce
\begin{align*}
V(f) & =E\left[\left(\int_{0}^{\infty}f(X(t))\,\mathrm{d}t-E\left[\int_{0}^{\infty}f(X(t))\,\mathrm{d}t\right]\right)^2\right]\\
 & =2\int_{\X}\int_{\X}f(y)f(y+z)\G(\mathrm{d}y)\G(\mathrm{d}z)-\int_{\X}\int_{\X}f(y)f(z)\G(\mathrm{d}y)\G(\mathrm{d}z)\\
 & =f^{2}(0)+2\int_{\X}\int_{\X}f(y)f(y+z)G(y)G(z)\,\mathrm{d}y\,\mathrm{d}z-\int_{\X}\int_{\X}f(y)f(z)G(y)G(z)\,\mathrm{d}y\,\mathrm{d}z.
\end{align*}

\begin{proposition}
Assume that $a(x)$ is even in each variable:
$$
a(x_1,\dots,- x_k, \dots,x_d)=  a(x_1,\dots, x_k, \dots,x_d),\;\; k=1,\dots, d.
$$
Then for the function $f$ given in  \eqref{exp-function} holds $V(f)>0$.

\begin{proof}
For the regular part of the Green measure $G(x)$ we have as above
$$
G (x)= \frac{1}{(2\pi)^d} \int_{\X} \frac{ \hat{a}(k) e^{i(k,x)}}{1-\hat{a}(k)}\,\mathrm{d}k.
$$
Because $a(x)$ is  symmetric in each variable by assumption,
the function $G(x)$ has the same symmetry. Using this symmetry we reduce
the equality for $V(f)$ to the integration over positive octant: 
$$
V(f)= f^2(0) + 2\int_{\X_+} \int_{\X_+} f(y)f(y+z) G(y) G(z)\,\mathrm{d}y\,\mathrm{d}z  -\int_{\X_+} \int_{\X_+} f(y) f(z)  G(y) G(z)\,\mathrm{d}y\,\mathrm{d}z.
$$
Now after a change of variables the integral part of this formula becomes
$$
\varepsilon^d  \int_{\X_+} \int_{\X_+} e^{-\varepsilon\sum y_k -\sum z_k}
[2e^{-\varepsilon \sum y_k} -1] G(\varepsilon y) G(z)\,\mathrm{d}y\,\mathrm{d}z.
$$
In the last expression $G(y)$ is continuous at $0$ and the integrand monotonically growing
for $\varepsilon \to 0$ with the point-wise limiting function with infinite integral. Then for some
$\varepsilon$ this expression is positive. Note that this expression in fact does not depend on $\varepsilon$.
\end{proof}

\end{proposition}

\subsection{Particular models}

The main technical question is  a bound for $a_k(x)$ in $k$ and $x$ together for the analysis of the properties
of $G(x)$.  From the stochastic point of view,  $a_k(x)$ is the density
of sum of  $k$  i.i.d. random variables with distribution density $a(x)$. Unfortunately, we could not find
in the literature any general result in this direction. There are several particular classes of jump kernels
for which we shall expect such kind of results. We will consider two examples, see \cite{KS} for details.

\begin{example}[Gauss kernels]

Assume that the jump kernels has the following form:

\begin{equation}
\label{G}
a(x)= C\exp\left(-\frac{ b |x|^2}{2}\right),\quad x\in \X.
\end{equation}

\begin{proposition}[\cite{KS}]
For the kernel (\ref{G}) and $d\geq 3$ holds
$$
 G_0(x)     \leq C_1 \exp\left(-\frac{ b|x|^2}{2}\right).
$$
\end{proposition}

\end{example}

\begin{example}[Exponential tails] 

Assume
\begin{equation}
\label{exp}
a(x)\leq C\exp(-\delta|x|), \quad x\in \X.
\end{equation}

\begin{proposition}[\cite{KS}] 
For the kernel (\ref{exp}) and $d\geq 3$ holds
$$
G_0(x)\leq A\exp(-B|x|)
$$
with certain $A,B>0$.
\end{proposition}

\end{example}

These examples show that for concrete Markov processes the regular component of the random Green
measure may have quick decay in the space
variable. This give us the possibility to use a larger class of 
admissible functions $f$.

\section{Brownian motion}
Let us consider another concrete example of Markov process.
Namely, denote $B(t)$, $t\ge 0$ the Brownian motion in $\X$ starting from the point $x$. The generator of
this process is the Laplace operator $\Delta$ considered in a proper
Banach space $E$.  As above we are interested in studying the random
variable
$$
Y(f) =\int_0^\infty f(B(t))\,\mathrm{d}t
$$
for certain class of functions $f:\X\to \R$.

\begin{theorem}
Let $d\geq  3$ and $x\in \X$ be given. The mapping
$$
Y: CL(\X) \to L^1 (\Omega, P)
$$
is a continuous linear operator and for all $f\in CL(\X)$  and every $\omega\in\Omega$ it 
has a unique representation
\begin{equation}
\label{reprB}
Y(f)(\omega)= \int_{\X} f(y) \mu^x (\mathrm{d}y, \omega)
\end{equation}
with a vector valued  $\sigma$-additive
$\mathrm{(}$in the strong topology of $L^1(\Omega, P)$$\mathrm{)}$
Radon measure $\mu^x (\mathrm{d}y, \omega)$ on $\B_b(\X)$.
\end{theorem}

\begin{proof}
The proof is essentially similar to the proof of Theorem~\ref{thm:representation1}.
Note that due to (\ref{GM}) we have
$$
E^x[Y(f)] =  - \Delta^{-1}f(x) = \int_{\X} \frac{f(y)}{|x-y|^{d-2}}\,\mathrm{d}y.
$$
Then
\begin{align*}
\left|\, \int_{\X} \frac{f(y)}{|x-y|^{d-2}} dy\right| &\leq \left|\, \int_{|x-y|\leq 1} \frac{f(y)}{|x-y|^{d-2}} dy\right|
+ \left|\, \int_{|x-y|>1} \frac{f(y)}{|x-y|^{d-2}} dy\right|\\
& \leq C_1 \|f\|_\infty + C_2 \|f\|_1 \leq C\|f\|_{CL}.
\end{align*}
In the last but one inequality we used the local integrability of $|x-y|^{2-d}$ in $y$, hence, 
$$
\|Y(f)\|_{L^1(P)} \leq C\|f\|_{CL}.
$$
This give us the possibility to apply the same arguments as in Theorem~\ref{thm:representation1} and the  representation (\ref{reprB}) follows.
\end{proof}

\section{Markov processes}

 Let  $X(t), t\ge 0$ be a Markov process in $\X$.
A standard way to define a   homogeneous  Markov process is to give the probability
$P_t(x,B)$ of the transition from the point $x\in \X$  to the set
$B\subset \X$ in   time $t>0$.  In some cases we have
$$
P_t(x,B)=\int_{B} p_t(x,y)\,{\mathrm{d}y},
$$
where $p_t(x,y)$ is the density of the transition probability. In any case, formally applying Fubini theorem, we obtain
\begin{equation}
\begin{gathered}\label{markov}
E\left[\int_0^\infty f(X(t))\,{\mathrm{d}t}\right]=\int_0^\infty E[f(X(t))]\,{\mathrm{d}t}=\int_0^\infty (T_tf)(x)\,{\mathrm{d}t}=\int_0^\infty \int_{\X}f(y)P_t(x,dy)\,{\mathrm{d}t}\\=
\int_{\X}f(y)\G(x, dy),
\end{gathered}
\end{equation}
  where $\G(x,A)=\int_0^\infty P_t(x,A)\,{\mathrm{d}t}$ is a Green measure of the process $X$, see \cite{KS}. If the density of the transition probability exists, then we can consider the Green function $$g(x,y) =\int_0^\infty p_t(x,y)\,{\mathrm{d}t},$$ and in this case formally

\begin{equation} \label{markov1} 
E\left[\int_0^\infty f(X(t))\,{\mathrm{d}t}\right]=\int_{\X}f(y)g(x,y)\,{\mathrm{d}y}.\end{equation}  Under certain  conditions   on $g(x,y)$   we can check the existence of the right-hand side of \eqref{markov1} and consequently,  the perpetual functional 
or the random potential $\int_0^\infty f(X(t))\,{\mathrm{d}t}.$ Note that the examples considered before as Brownian motion and compound Poisson process, being Markov processes, provide such examples.

Consider an example of a Markov process without independent increment.  Let 

$$
Lu(x) = \sum_{k,j=1}^d   \partial_{x_k} a_{k,j}(x) \partial_{x_j} u(x)
$$
be a uniformly elliptic differential operator with a symmetric matrix $(a_{k,j})$ in the divergent form. By Aronson's theorem \cite{Aron} 
its heat kernel $p_t(x,y)$ (equivalently, the transition density of the diffusion process $X(t)$
generated by $L$) satisfies the two-side Gaussian bound:
\begin{equation}
\label{Aron}
\frac{C_{-} }{t^{d/2}} \exp\left(- c_{-} \frac{|x-y|^2}{t}\right)\leq p_t(x,y)\leq \frac{C_+}{t^{d/2} }\exp\left(- c_{+} \frac{|x-y|^2}{t}\right)
\end{equation}
Using this fact similarly to the case of the Brownian motion
we can show the existence of the random potential
$$
Y(f)=\int_0^\infty f(X(t))\,{\mathrm{d}t}
$$
for all $f\in CL(\X)$. This shows the following theorem.

\begin{theorem}
Let $d\geq  3$ and $x\in \X$ be given. In addition, let $X(t)$, $t\ge 0$. be a Markov process such that the transition density satisfies \eqref{Aron}. The mapping
$$
Y: CL(\X) \to L^1 (\Omega, P)
$$
is a continuous linear operator and for all $f\in CL(\X)$  and every $\omega\in\Omega$ it
has a unique representation
\begin{equation}
\label{reprC}
Y(f)(\omega)= \int_{\X} f(y) \mu^x (\mathrm{d}y, \omega)
\end{equation}
with a vector valued  $\sigma$-additive
$\mathrm{(}$in the strong topology of $L^1(P)$$\mathrm{)}$
Radon measure $\mu^x (\mathrm{d}y, \omega)$ on $\B_b(\X)$.
\end{theorem}

\end{document}